\documentclass[11pt]{amsart}
\usepackage{graphicx}
\usepackage[active]{srcltx}
 \makeatletter
\renewcommand*\subjclass[2][2000]{%
  \def\@subjclass{#2}%
  \@ifundefined{subjclassname@#1}{%
    \ClassWarning{\@classname}{Unknown edition (#1) of Mathematics
      Subject Classification; using '1991'.}%
  }{%
    \@xp\let\@xp\subjclassname\csname subjclassname@#1\endcsname
  }%
}
 \makeatother
\usepackage{enumerate,url,amssymb,  mathrsfs}

\newtheorem{theorem}{Theorem}[section]
\newtheorem{lemma}[theorem]{Lemma}
\newtheorem*{lemma*}{Lemma}
\newtheorem{proposition}[theorem]{Proposition}
\newtheorem{corollary}[theorem]{Corollary}

\theoremstyle{definition}
\newtheorem{definition}[theorem]{Definition}

\newtheorem{conjecture}[theorem]{Conjecture}

\theoremstyle{remark}
\newtheorem{remark}[theorem]{Remark}

\numberwithin{equation}{section}

\DeclareMathOperator*{\esssup}{ess\,sup}

\def\XXint#1#2#3{{\setbox0=\hbox{$#1{#2#3}{\int}$}
\vcenter{\hbox{$#2#3$}}\kern-.5\wd0}}

\begin{document}

\title{{Boundedness of the orthogonal projection on Harmonic Fock spaces}}



 \author{Djordjije Vujadinovi\'c}
\address{ University of Montenegro, Faculty of Mathematics, Dzordza Va\v singtona  bb, 81000 Podgorica, Montenegro}
 \email{djordjijevuj@t-com.me}
 \date{}
\begin{abstract}
The main result of this paper refers to the  boundedness of the orthogonal projection $P_{\alpha}:L^{2}(\mathbb{R}^{n},d\mu_{\alpha})\rightarrow \mathcal{H}_{\alpha}^{2}, n\geq2 $ associated to the harmonic Fock space $\mathcal{H}_{\alpha}^{2},$ where $d\mu_{\alpha}(x)=(\pi\alpha)^{-n/2}e^{-\frac{|x|^2}{\alpha}}dx.$ We prove that the operator $P_{\alpha}$  is not bounded on $L^{p}(\mathbb{R}^{n},d\mu_{\beta})$ when $0<p< 1$ and we found a necessary and sufficient condition for the boundedness when $1\leq p<\infty$ and $n$ is an even integer.
 \end{abstract}

\maketitle

\section{ Introduction and Preliminaries}

\subsection{ The harmonic Fock space}

\indent The Segal–-Bargmann space, denoted by $\mathcal{F}_{\alpha},$ also known as the Fock space consists of all entire functions on $\mathbb{C}^{n}$ which are square-integrable with respect to the Gaussian measure $$d\mu_{\alpha}(z)=\frac{1}{(\pi \alpha)^n}e^{-|z|^2}dz,\alpha>0,$$ where $dz$ is the Lebesgue measure on  $\mathbb{C}^{n}.$ The reproducing kernel in this case is given by
$$K_{\alpha}(x,z)=e^{\frac{\left<x,z\right>}{\alpha}},\enspace  x,z\in\mathbb{C}^{n}. $$
Thus,
$$f(x)=\left<f(\cdot),K_{\alpha}(\cdot,x)\right>=\int_{\mathbb{C}^{n}}K_{\alpha}(x,z)f(z)d\mu_{\alpha}(z),$$
where $f\in \mathcal{F}_{\alpha}, \enspace
x\in \mathbb{C}^{n}$ (see \cite{foland}).

Here, we should  mention the existing extensive study on the Fock analytic spaces in the planar case ($n=2$) presented in \cite{zhu} and \cite{janson} for the higher dimensional case of the analytic Fock spaces.

In \cite {janson} authors considered Hankel forms on the Hilbert space of analytic functions square integrable with respect to the given measure on certain domain in $\mathbb{C}^{n}.$ They obtained necessary and sufficient conditions for boundedness, compactness and belonging to the Schatten classes $S_{p},p\geq 1,$ for Hankel forms.  In the chapter 7. in \cite{janson} where the analytic Fock spaces were introduced, the mentioned general theory was applied to the Fock space. In this particular case a lot of results were obtained such as: interpolation, atomic decomposition, boundednees of the orthogonal projections,  characterization of dual spaces etc.

In this paper we are interested in a "harmonic" analogue of the mentioned analytic case. Namely, the harmonic Fock space $\mathcal{H}_{\alpha}^{2}$ is defined to be the space of all harmonic functions which belong to $L^{2}(\mathbb{R}^{n}, d\mu_{\alpha}),$
$$\mathcal{H}_{\alpha}^{2}=\{f\in L^{2}(\mathbb{R}^{n}, d\mu_{\alpha})|\triangle f=0\},$$
where $d\mu_{\alpha}(y)=(\pi \alpha)^{-n/2} e^{-\frac{|y|^2}{\alpha}}dy$ is a probability measure on $\mathbb{R}^{n}$ and $n\geq 2.$

Following \cite{zhu}, the definitions, Definition \ref{normalna} and Definition \ref{norma1} appear as natural extension of the harmonic Fock space $\mathcal{H}_{\alpha}^{p}$ for general $p>0.$

\begin{definition}
\label{normalna}
Suppose $\alpha>0$ and $p>0.$ The space $L_{\alpha}^{p}$ is defined to be the space of all Lebesgue measurable functions $f$ such that
the function $f(x)e^{-\frac{|x|^2}{2\alpha}}$ is in $L^{p}(\mathbb{R}^n,d\mu_{\alpha}).$ The semi-norm $\|\cdot\|_{p,\alpha}$ is defined in the following way
\begin{equation}
\label{norma1}
\|f\|_{p,\alpha}^{p}=\left(\frac{p}{2\alpha\pi}\right)^{n/2}\int_{\mathbb{R}^{n}}\left|f(x)e^{-\frac{|x|^2}{2\alpha}}\right|^{p}dx.
\end{equation}
Specially, for $\alpha>0$ and $p=\infty,$ we write $L_{\alpha}^{\infty}$ to define the space of all Lebesgue mesurable functions in $\mathbb{R}^{n}$ such that
\begin{equation}
\label{norma2}
\|f\|_{\infty,\alpha}=\esssup\{|f(x)|e^{-\frac{|x|^2}{2\alpha}}:x\in \mathbb{R}^n\}<\infty.
\end{equation}
\end{definition}
\begin{definition}
\label{normalna1}
Let $\alpha>0,p>0.$ By $\mathcal{H}_{\alpha}^{p}$ we denote the space of all harmonic functions in $L_{\alpha}^{p}.$
\end{definition}

 The space $\mathcal{H}_{\alpha}^{p}$ is closed in $L_{\alpha}^{p}$ and for $0<p<1$ with the semi-norm \eqref{norma1} is a complete metric space, while for
$1\leq p\leq \infty$ with the norm \eqref{norma1} and \eqref{norma2} respectively is a Banach space.

According to the previous notation, the measure associated to the Fock space $\mathcal{H}_{\alpha}^{p}$ is $d\mu_{2\alpha/p}.$

\subsection{ The kernel $H_{\alpha}(x,y)$}

The problem of finding the exact formula for the reproducing kernel associated to $\mathcal{H}_{\alpha}^{2}$ for $n\geq3$ was resolved in \cite{miroslav}. Author M.Engli\v{s} considered the asymptotic expansion for the Berezin transform related to the harmonic Fock space. The starting result in this research was the computation of  the reproducing kernel. It turns out that the kernel is presented by one of Horn's hypergeometric function of two variables $\Phi_{2}$ (see Horn's list in \cite{bateman})  which is entire function on $\mathbb{C}^2$ defined as
$$\Phi_{2}\left({a,b\atop c};z,w\right)=\sum_{j,k=0}^{\infty}\frac{(a)_{j}(b)_{k}}{(c)_{j+k}j!k!}z^{j}w^{k},\enspace z,w\in \mathbb{C}.$$
Here, $(a)_{k},$ as usual, denotes the Pochhammer symbol, $$(a)_{k}=a(a+1)(a+2)\cdot\cdot\cdot(a+k-1)=\frac{\Gamma(a+k)}{\Gamma(a)}.$$

 Explicitly, the following result was established (see \cite{miroslav}, pp. 6).
\begin{proposition}
\label{cheh}
  The harmonic Fock kernel $H_{\alpha}(x,y)$ is given by
  $$H_{\alpha}(x,y)=\Phi_{2}\left(\frac{n}{2}-1;\frac{n}{2}-1;\frac{n}{2}-1\mid\frac{t_1+it_2}{\alpha}, \frac{t_1-it_2}{\alpha}\right),$$
  where $t_1=\left<x,y\right>, t_2=\sqrt{|x|^2|y|^2-\left<x,y\right>^2}.$
\end{proposition}

The reproducing property of the kernel $H_{\alpha}(x,y)$ yields the following integral representation
$$f(x)=\int_{\mathbb{R}^{n}}H_{\alpha}(x,y)f(y)d\mu_{\alpha}(y),\enspace f\in \mathcal{H}_{\alpha}.$$
Further, we want to present  another formula of the kernel $H_{\alpha}(x,y)$
which would be more convenient for certain computation operations. For this purpose, let us recall some basic facts devoted to the zonal harmonics (for more details see \cite{Axler}).

In the sequel, by $P(\mathbb{R}^{n})$ we denote a set of all polynomials in $\mathbb{R}^{n}$  and ${\it P}_{m}(\mathbb{R}^{n})$ denotes a set of all homogenous polynomials of degree $m\in \mathbb{N},$ while ${\it H}_{j}(\mathbb{R}^{n})$ is a set of all harmonic homogenous polynomials of degree $j.$  The restriction of any polynomial from  ${\it H}_{j}(\mathbb{R}^{n})$ to $\mathbb{S}^{n-1}$ is known as a spherical harmonic of degree $j$ and the collection of all spherical harmonics of degree $j$ we denote by ${\it H}_{j}(\mathbb{S}^{n-1}).$ The space ${\it H}_{k}(\mathbb{S}^{n-1})$ is orthogonal to ${\it H}_{j}(\mathbb{S}^{n-1})$ if $j\neq k$ regarding the inner product in $L^{2}(\mathbb{S}^{n-1})$  defined as
$$\left<f,g\right>=\int_{\mathbb{S}^{n-1}}f(\xi)\overline{g(\xi)}d\sigma'(\xi), \enspace f,g\in L^{2}(\mathbb{S}^{n-1}).$$
 Here,  the surface measure $d\sigma'$ on $\mathbb{S}^{n-1}$ is normalized.

Moreover, any homogenous polynomial $p\in{\it P}_{m}(\mathbb{R}^{n})$ can be represented as a unique sum of harmonic homogenous polynomials (see Theorem 5.7 in \cite{Axler}, page 77), i.e.,
\begin{equation}
\label{polinom1}p=p_{m}+|x|^{2}p_{m-2}+\cdot\cdot\cdot+|x|^{2k}p_{m-2k},\end{equation}
where $k=[\frac{m}{2}]$ and each $p_{j}\in {\it H}_{j}(\mathbb{R}^{n}).$

The above decomposition reduces on
\begin{equation}
\label{polinom11}p=p_{m}+p_{m-2}+\cdot\cdot\cdot+p_{m-2k},\end{equation}
on $\mathbb{S}^{n-1}.$

The spherical harmonic $\mathcal {Y}_{m}(\cdot,\eta)\in {\it H}_{m}(\mathbb{S}^{n-1}),$ where $\eta\in\mathbb{S}^{n-1}$ is fixed point, is called the zonal harmonic of degree $m$ with pole $\eta$ if
$$p(\eta)=\int_{\mathbb{S}^{n-1}}p(\xi)\overline{\mathcal{Y}_{m}(\xi,\eta)}d\sigma'(\xi), \enspace p\in {\it H}_{m}(\mathbb{S}^{n-1}).$$

The explicit formula for the zonal harmonic of degree $m$  with a pole $\xi\in \mathbb{S}^{n-1},$ ${\mathcal{ Y}}_{m}(\cdot,\xi)$ is given by the following formula
\begin{equation}
\begin{split}&{\mathcal {Y}}_{m}(x,\xi)\\
&=(n+2m-2)\sum_{k=0}^{[m/2]}(-1)^{k}\frac{n(n+2)...(n+2m-2k-4)}{2^{k}k!(m-2k)!}\left<x,\xi\right>^{m-2k}|x|^{2k},\end{split}\end{equation}
where $x\in \mathbb{R}^{n}$ and  $m>0.$

Another way to present the kernel $H_{\alpha}(x,y)$ can be expressed as follows (see the proof of the Proposition \ref{cheh} in \cite{miroslav})
\begin{equation}
\label{red}
H_{\alpha}(x,y)=\sum_{k=0}^{\infty}\frac{Y_{k}(x,y)}{\alpha^k(\frac{n}{2})_{k}},
\end{equation}
where $Y_{m}(x,y)=|y|^{m}{\mathcal{Y}_{m}}(x,\frac{y}{|y|}).$
Clearly, the restriction of $Y_{m}(\cdot,\frac{y}{|y|})$ on $\mathbb{S}^{n-1}$ is the zonal harmonic ${\mathcal{Y}_{m}}.$

Specially, if  we consider the planar case when $n=2,$ the zonal harmonics ${\mathcal Y}_{m}$ are given by (see \cite{Axler}, pp.94)
\begin{equation}
\label{zonal2}
{\mathcal Y}_{m}(e^{i\theta},e^{i\varphi})=2\cos{m(\theta-\varphi)},
\end{equation}
i.e. considering arbitrary vectors $x,y\in \mathbb{R}^{2}$ as complex numbers $x=z, y=w$ where $z=|z|e^{i\theta}$ and $w=|w|e^{i\varphi},$ we have
\begin{equation}
\label{zonal21}
 Y_{m}(x,y)=Y_{m}(z,w)=|z|^{m}|w|^{m}{\mathcal Y}_{m}(e^{i\theta},e^{i\varphi}).
\end{equation}
Now, the representation \eqref{red} implies
\begin{equation}
\begin{split}
\label{red1}
H_{\alpha}(x,y)&=H_{\alpha}(z,w)\\
&=2\sum_{k=0}^{\infty}\frac{|z|^{k}|w|^{k}}{\alpha^{k}\Gamma(k+1)}\cos(k(\theta-\varphi))\\
&=2\cos\left(\frac{|z||w|}{\alpha}\sin(\theta-\varphi)\right) e^{\frac{|z||w|\cos(\theta-\varphi)}{\alpha}}\\
&=2e^{\frac{\left<x,y\right>}{\alpha}} \cos\left(\frac{|x||y|}{\alpha}\sin{t}\right), \\
\end{split}
\end{equation}
where $\cos{t}=\frac{\left<x,y\right>}{|x||y|}.$ Therefore, $H_{\alpha}(x,x)=2e^{\frac{|x|^2}{\alpha}},$ while $H_{\alpha}(x,x)=1,$ for $x=0.$

Further, it is interesting to note that the special case $n=4$ gives an explicit formula for the kernel $H_{\alpha}(x,y).$

Namely, according to the Proposition \ref{cheh} for $n=4$ we have
$$H_{\alpha}(x,y)=e^{\frac{\left<x,y\right>}{\alpha}}\frac{\sin(t+\sin(\frac{|x||y|\sin{t}}{\alpha}))}{\sin(t)}, t=\frac{\left<x,y\right>}{|x||y|}$$
while
\begin{equation*}
H_{\alpha}(x,x)=e^{\frac{|x|^2}{\alpha}}\left(1+\frac{|x|^2}{\alpha}\right), x\in\mathbb{R}^{4}.
\end{equation*}

In fact, for any $n\geq3$ we have a "trace" formula for the kernel $H_{\alpha}(x,y).$

Explicitly,
\begin{equation}
\label{x=y}
H_{\alpha}(x,x)={}_1F_{1}\left(-2+n,-1+\frac{n}{2};\frac{|x|^2}{\alpha}\right), \enspace x\in\mathbb{R}^{n}.
\end{equation}
Here, by ${}_1F_{1}(a,b;x)$  we denote the confluent hypergeometric function  or Kummer's function of the first kind.  For a comprehensive study on the confluent hypergeometric functions we refer to \cite{andrew} (Chapter 4) and \cite{abramovits}(Chapter 13).

The important property of the confluent hypergeometric functions which will be used in proving the main results of this paper is related to the asymptotic behaviour of the function $ {}_1F_{1}(a,b;x)$ for a large argument, and it is given by

\begin{equation}
\label{ast1234}
{}_1F_{1}(a,b;x)\sim\frac{\Gamma(b)e^{x}}{\Gamma(a)x^{b-a}}{}_2F_{0}\left({b-a,1-a\atop-};\frac{1}{x}\right), x\rightarrow+\infty
\end{equation}
(see \cite{andrew}, pp. 193).

Using the reproducing property of the kernel $H_{\alpha},$ it is not hard to check that the following identity holds

\begin{equation}
\label{kernel}
\int_{\mathbb{R}^n}|H_{\alpha}(x,y)|^{2}d\mu_{\alpha}(y)={}_1F_{1}\left(-2+n,-1+\frac{n}{2};\frac{ |x|^2}{\alpha}\right).
\end{equation}
Further, for  $\alpha,\beta>0$ and  $1\leq q\leq2,$   applying the Minkowski inequality we obtain

\begin{equation}
\label{otkrice1112}
\int_{\mathbb{R}^{n}}|H_{\alpha}(x,y)|^{q}d\mu_{\alpha}(y)\leq {}_1F_{1}^{q/2}\left(-2+n,-1+\frac{n}{2};\frac{ |x|^2}{\alpha}\right).
\end{equation}
Inserting $\frac{\alpha x}{\beta}$ in \eqref{otkrice1112} instead of $x$ we get
\begin{equation}
\label{otkrice}
\int_{\mathbb{R}^{n}}|H_{\beta}(x,y)|^{q}d\mu_{\alpha}(y)\leq {}_1F_{1}^{q/2}\left(-2+n,-1+\frac{n}{2};\frac{ \alpha|x|^2}{\beta^{2}}\right).
\end{equation}

Using the relation \eqref{ast1234} we can describe the asymptotic behaviour of the function $H_{\alpha}(x,x)$ for a large argument which is given by
\begin{equation}
\label{ast}
H_{\alpha}(x,x)\sim\frac{\alpha^{1-\frac{n}{2}}\Gamma(-1+\frac{n}{2})e^{x}}{\Gamma(-2+n)|x|^{2-n}}{}_2F_{0}\left({1-\frac{n}{2},n-3\atop-};\frac{\alpha}{|x|^2}\right), |x|\rightarrow+\infty.
\end{equation}

The analogous asymptotic formula for  $H_{\alpha}(x,y)$ is given by the certain "substitute" formula (see (17) in \cite{miroslav}, pp.10)
\begin{equation}
\label{kontura}
H_{\alpha}(x,y)=\frac{i}{\pi}\frac{n-2}{2^{n-1}}\oint_{\gamma}a^{n-3}(a-1)^{-\frac{n}{2}}\int_{-1}^{1}(1-t^2)^{\frac{n}{2}-2}e^{a\frac{\left<x,y\right>+itV(x,y)}{\alpha}}dtda.
\end{equation}
Here $\gamma$ is the contour in the complex plane cut along the real axis from $-\infty$ to 1, and $\gamma$ "goes" from 0 to $1-\epsilon$ along the "upper" edge of the cut, then around 1 clockwise, and "returns" from $1-\epsilon$ to 0 along the "lower"edge. The formula \eqref{kontura}
 gives the explicit formula of the kernel $H_{\alpha}(x,y)$ when the dimension $n$ is an even number.
 Precisely, for $n=2N+2,N\geq 1$ we have
\begin{equation}
\label{miroslav}
\begin{split}
H_{\alpha}(x,y)=\frac{2^{1-2N}}{(N-1)!}&\left(\frac{\alpha}{iV}\right)^{N-1}\sum_{j=0}^{N}{\left(N\atop j\right)}\frac{N!}{j!}\\
&\int_{-1}^{1}G_{N}(t)\left(\frac{\left<x,y\right>+itV}{\alpha}\right)^{j}e^{\frac{\left<x,y\right>+itV}{\alpha}}dt,
\end{split}
\end{equation}
where $$G_{N}(t)=(-1)^{N-1}\frac{\partial^{N-1}}{\partial t^{N-1}}(1-t^2)^{N-1}.$$

\subsection{ The orthogonal projection $P_{\alpha}$}

As it was stated in the introduction part the orthogonal projection
$$P_{\alpha}: L^{2}(\mathbb{R}^{n},d\mu_{\alpha})\rightarrow \mathcal{H}_{\alpha}^{2}$$
is an integral operator defined as
$$P_{\alpha}f(x)=\int_{\mathbb{R}^{n}}H_{\alpha}(x,y)f(y)d\mu_{\alpha}(y).$$

Clearly, the operator $P_{\alpha}$ is bounded on $L^{2}(\mathbb{R}^{n},d\mu_{\alpha})$ as an orthogonal projection and its norm $\|P_{\alpha}\|_{L^{2}(\mathbb{R}^{n},d\mu_{\alpha})}=1.$

Through the paper we will consider the operator $P_{\alpha}$ on $L^{p}(\mathbb{R}^{n},d\mu_{\beta})$ in the following form
\begin{equation}
\label{konjg11}
P_{\alpha}f(x)=\left(\frac{\beta}{\alpha}\right)^{\frac{n}{2}}\int_{\mathbb{R}^n}H_{\alpha}(x,y)e^{(\frac{1}{\beta}-\frac{1}{\alpha})|y|^2}f(y)d\mu_{\beta}(y),
\end{equation}
which will be more convenient for the further observations.

\subsection{ The main result}

The investigation of the $L^{p}-$boundedness as well as the estimation of the norm for a class of integral operators induced by the reproducing Fock kernel was considered in \cite{domi1}. Also, some similar question in a different framework (the Bergman space) was treated in a numerous papers (see for instance \cite{liu},\cite{kajal},\cite{domi}). In \cite{domi} M.Dostani\'c  showed that the Bergman projection is not bounded on  $L^{p}$ space with the certain exponential weight for $p\neq 2.$

Authors  in \cite{domi1} observed $n-$dimensional complex space $\mathbb{C}^{n}$ with the Gaussian probability measure $dv_{t}(z)=\left(\frac{t}{\pi}\right)^{n}e^{-t|z|^2}dv(z),t>0,$ where $dv$ is Lebes--gue measure on $\mathbb{C}^{n}.$ The analytic Fock space $F_{t}^{p}$ consists from all entire functions in $L^{p}(\mathbb{C}^{n},dv_{t}),p>0.$ The reproducing kernel associated to $F_{t}^{2}$ is $K_{t}(x,y)=e^{t\left<x,y\right>}.$

The operators under the consideration in \cite{domi1} are
$$S_{t}f(z)=\int_{\mathbb{C}^{n}}e^{t\left<z,w\right>}f(w)dv_{t}(w)$$
and
$$T_{t}f(z)=\int_{\mathbb{C}^{n}}|e^{t\left<z,w\right>}|f(w)dv_{t}(w), $$ $f\in L^{p}(\mathbb{C}^{n},dv_{t}).$

The main result of \cite{domi1} deals with the problem of boundedness  of operators $S_{t}$ and $T_{t}$ on $L^{p}(\mathbb{C}^{n},dv_{s})$ for fixed parameters $s>0, t>0.$
In fact, it is proved that $T_{t}$ ($S_{t}$) is bounded on $L^{p}(\mathbb{C}^{n},dv_{s})$ if and only if $pt=2s.$

A natural and important question which arises in the context of harmonic Fock spaces is to determine the conditions under which the operator $P_{\alpha}$ is bounded on $L^{p}(\mathbb{R}^{n},d\mu_{\beta})$ for various choices of parameter $\beta>0.$

The following theorem presents the main result of this paper and it gives an  answer  to the previously raised question under certain conditions.
\begin{theorem}
\label{main}
Suppose $\alpha>0,\beta>0$ and $0<p<\infty.$ Then:\\
a) The operator $P_{\alpha}$ is not bounded on $L^{p}(\mathbb{R}^{n},d\mu_{\beta})$ for $0<p<1.$\\
b) For $p\geq 1$ and $n=2N,N\geq1$, the operator $P_{\alpha}$ is  bounded on $L^{p}(\mathbb{R}^{n},d\mu_{\beta})$
if and only if $$p\beta=2\alpha.$$
\end{theorem}
As the matter of fact, we prove that $p\beta=2\alpha$ is a necessary condition for the boundedness of the operator $P_{\alpha}$ for any $n\geq2.$

As an immediate consequence we obtain the following result.
\begin{corollary}
 Let $n=2N,$ $N\geq1$ and $\alpha>0.$ For $p\geq 1$ the operator $P_{\alpha}$ is a bounded projection from $L_{\alpha}^{p}$ onto $\mathcal{H}_{\alpha}^{p}.$
\end{corollary}
It remains an open problem to extend the Theorem \ref{main} for the case of an arbitrary dimension $n\geq 3.$
It seems that the main problem is produced from the fact that the formula \eqref{kontura}  provides an explicit formula \eqref{miroslav} only for even number $n.$
However, we conjecture the sequel extended result.
\begin{conjecture}
Suppose $\alpha>0,\beta>0$ and $0<p<\infty.$ Then:\\
a) The operator $P_{\alpha}$ is not bounded on $L^{p}(\mathbb{R}^{n},d\mu_{\beta})$ for $0<p<1.$\\
b) For $p\geq 1$ the operator $P_{\alpha}$ is  bounded on $L^{p}(\mathbb{R}^{n},d\mu_{\beta}) \Leftrightarrow$ $p\beta=2\alpha.$
\end{conjecture}
In the sequel, the dimension $n$ is considered to be the integer such that $n\geq 2,$ unless stated otherwise.
\section{Proof of the Theorem \ref{main}}
The section is organized as follows: In lemmas, Lemma \ref{lema1} and Lemma \ref{vaznalema}  we derive a necessary condition for the boundedness of the operator $P_{\alpha}$ when $1\leq p<\infty.$ In theorems,  Theorem \ref{dovoljan} and Theorem \ref{dovoljan1}  we found the sufficient condition for the boundedness of the operator $P_{\alpha}$ on $L^{p}(\mathbb{R}^{n},d\mu_{\beta})$ when $p\geq1,$ with the restriction $n=2N,$ $N\geq1.$
\begin{lemma}
\label{lema1}
Suppose $0<p<\infty$ and $P_{\alpha}$ is bounded on $L^{p}(\mathbb{R}^{n},d\mu_{\beta}),$ then\\
$p\beta\leq 2\alpha$ when  $1\leq p<\infty$
 and  $P_{\alpha}$ is unbounded for $0<p<1.$
\end{lemma}
\begin{proof}
Let $f_{z,k}^{x}(y)=e^{-x|y|^2}|y|^{k} \mathcal{Y}_{k}(z,\frac{y}{|y|}),y\in \mathbb{R}^{n},$ $x>0.$  Here,  $z$ is some fixed vector from $\mathbb{R}^n.$ Using the polar coordinates we get
\begin{equation}
\begin{split}
&\int_{\mathbb{R}^n}|f_{z,k}^{x}(y)|^{p}d\mu_{\beta}(y)\\
&=(\pi\beta)^{-\frac{n}{2}}\int_{0}^{\infty}e^{-\frac{r^2}{\beta}}r^{n-1}dr\int_{\mathbb{S}^{n-1}}|f_{z,k}^{x}(r\xi)|^{p}d\sigma(\xi)\\
&=(\pi\beta)^{-\frac{n}{2}}|z|^{pk}\int_{0}^{\infty}r^{pk+n-1}e^{-(px+\frac{1}{\beta})r^2}dr\int_{\mathbb{S}^{n-1}}\left|\mathcal{Y}_{k}\left(\frac{z}{|z|},\xi\right)\right|^p d\sigma(\xi)\\
&=(\pi\beta)^{-\frac{n}{2}}\frac{A_{k}^{p}(z)|z|^{pk}\Gamma(\frac{pk+n}{2})}{2(px+\frac{1}{\beta})^{\frac{pk+n}{2}}},
\end{split}
\end{equation}
where $A_{k}^{p}(z)=\int_{\mathbb{S}^{n-1}}|\mathcal{Y}_{k}(\frac{z}{|z|},\xi)|^p d\sigma(\xi).$

On the other hand,
\begin{equation}
\begin{split}
&(P_{\alpha}f_{z,k}^{x})(w)=\int_{\mathbb{R}^n}H_{\alpha}(w,y)f_{z,k}^{x}(y)d\mu_{\alpha}(y)\\
&=\sum_{i=0}^{\infty}\frac{1}{\alpha^i(\frac{n}{2})_{i}}\int_{\mathbb{R}^n}Y_{i}(w,y)f_{z,k}^{x}(y)d\mu_{\alpha}(y)\\
&=\sum_{i=0}^{\infty}\frac{|w|^{i}|z|^i}{(\pi\alpha)^{n/2}\alpha^i(\frac{n}{2})_{i}}\int_{0}^{\infty}e^{-(x+\frac{1}{\alpha})r^2}r^{k+i+n-1}dr\\
&\times\int_{\mathbb{S}^{n-1}}\mathcal{Y}_{i}\left(\frac{w}{|w|},\xi\right)\mathcal{Y}_{k}\left(\frac{z}{|z|},\xi\right)d\sigma(\xi)\\
&=\frac{|z|^{k}}{(\alpha x+1)^{k+\frac{n}{2}}}\mathcal{Y}_{k}\left(w,\frac{z}{|z|}\right).
\end{split}
\end{equation}
Therefore,
\begin{equation}
\int_{\mathbb{R}^n}|(P_{\alpha}f_{z,k}^{x})(w)|^p d\mu_{\beta}(w)=\left(\frac{|z|^{k}}{(\alpha x+1)^{k+\frac{n}{2}}}\right)^{p}\beta^{\frac{kp}{2}}\frac{A_{k}^{p}(z)\Gamma(\frac{pk+n}{2})}{2\pi^{n/2}}.
\end{equation}
If the operator $P_{\alpha}$ is bounded on $L^{p}(\mathbb{R}^2, d\mu_{\beta}),$ then it must exists a constant $C$ such that
\begin{equation}
\label{bitno1}
\left(\frac{|z|^{k}}{(\alpha x+1)^{k+\frac{n}{2}}}\right)^{p}\beta^{\frac{kp}{2}}\frac{A_{k}^{p}(z)\Gamma(\frac{pk+n}{2})}{2\pi^{n/2}}\leq C \frac{A_{k}^{p}(z)|z|^{pk}\Gamma(\frac{pk+n}{2})}{2(\pi\beta)^{\frac{n}{2}}(px+\frac{1}{\beta})^{\frac{pk+n}{2}}},
\end{equation}
i.e.,
$$\frac{1}{(\alpha x+1)^{pk+\frac{pn}{2}}}\leq \frac{C}{(p\beta x+1)^{\frac{pk+n}{2}}}.$$
 If we fix $x$ and take the $k-$th root in \eqref{bitno1}, by letting  $k\rightarrow +\infty$ we get
$$\frac{1}{(\alpha x+1)^{p}}\leq \frac{1}{(p\beta x+1)^{\frac{p}{2}}},$$
i.e. $p\beta x+1\leq (\alpha x+1)^2.$ Since $x> 0,$ we have that $2\alpha\geq p\beta.$
\end{proof}

Since the operator $P_{\alpha}$ is not bounded for $0<p<1,$  in the sequel we will consider  the case when $p\geq 1.$
\begin{lemma}
\label{vaznalema}
Suppose $1\leq p<\infty  $ and $P_{\alpha}$ is bounded on $L^{p}(\mathbb{R}^n, d\mu_{\beta}).$ Then $2\alpha=p\beta.$
\end{lemma}
\begin{proof}

{\bf The case $1<p\leq 2.$}

Once again, let us observe the function $f_{z,k}^{x}(y)=e^{-x|y|^2}|y|^{k}\mathcal{Y}_{k}\left(z,\frac{y}{|y|}\right)$ $x>0,$ $k$ is a positive integer  as it was before, and $z$ is a fixed vector in $\mathbb{R}^{n}.$

 Similarly to the previous calculations, where we used the polar coordinates and orthogonality of zonal harmonics, we have
\begin{equation}
\label{povratakotpisanih}
\begin{split}
P_{\alpha}^{\ast}f_{z,k}^{x}(w)&=\left(\frac{\beta}{\alpha}\right)^{\frac{n}{2}}e^{(\frac{1}{\beta}-\frac{1}{\alpha})|w|^2}\int_{\mathbb{R}^n}H_{\alpha}(w,y)f_{z,k}^{x}(y)d\mu_{\beta}(y)\\
&=(\pi\alpha)^{-n/2}e^{(\frac{1}{\beta}-\frac{1}{\alpha})|w|^2}\sum_{i=0}^{\infty}\int_{\mathbb{R}^n}\frac{Y_{i}(w,y)}{\alpha^{i}(\frac{n}{2})_i}|y|^{k}\mathcal{Y}_{k}\left(z,\frac{y}{|y|}\right)e^{-(x+\frac{1}{\beta})|y|^2}dy\\
&=e^{(\frac{1}{\beta}-\frac{1}{\alpha})|w|^2}\frac{|z|^{k}}{\alpha^{n/2+k}(x+\frac{1}{\beta})^{k+\frac{n}{2}}}\mathcal{Y}_{k}\left(w,\frac{z}{|z|}\right).\\
\end{split}
\end{equation}
On the other hand,
\begin{equation}
\int_{\mathbb{R}^n}|f_{z,k}^{x}(y)|^{q}d\mu_{\beta}(y)=A_{k}^{q}(z)\frac{|z|^{qk}\Gamma(\frac{qk+n}{2})}{2(\pi\beta)^{n/2}(qx+\frac{1}{\beta})^{\frac{qk+n}{2}}},
\end{equation}
where $A_{k}^{q}(z)=\int_{\mathbb{S}^{n-1}}|\mathcal{Y}_{k}(\frac{z}{|z|},\xi)|^{q}d\sigma(\xi).$

Since we suppose that $P_{\alpha}$ is bounded operator there is a constant $C$ such that
\begin{equation}
\label{nejdn}
\int_{\mathbb{R}^n}|P_{\alpha}^{\ast}f_{z,k}^{x}(w)|^{q}d\mu_{\beta}(w)\leq C \int_{\mathbb{R}^n}|f_{z,k}^{x}(w)|^{q}d\mu_{\beta}(w).
\end{equation}
Furthermore,
\begin{equation}
\begin{split}
&\int_{\mathbb{R}^n}|P_{\alpha}^{\ast}f_{z,k}^{x}(w)|^{q}d\mu_{\beta}(w)\\
&=(\pi\beta)^{-n/2}\frac{|z|^{qk}}{(\alpha x+\frac{\alpha}{\beta})^{qk+\frac{qn}{2}}}\int_{\mathbb{R}^n}\left|\mathcal{Y}_{k}\left(w,\frac{z}{|z|}\right)\right|^{q}e^{-(\frac{1}{\beta}-q(\frac{1}{\beta}-\frac{1}{\alpha}))|w|^2}dw\\
&=(\pi\beta)^{-n/2}\frac{A_{k}^{q}(z)|z|^{qk}}{2(\alpha x+\frac{\alpha}{\beta})^{qk+\frac{qn}{2}}}\frac{\Gamma(\frac{qk+n}{2})}{(\frac{1}{\beta}-q(\frac{1}{\beta}-\frac{1}{\alpha}))^{\frac{qk+n}{2}}}.\\
\end{split}
\end{equation}
Now, inequality \eqref{nejdn} becomes

$$\frac{1}{(\alpha x+\frac{\alpha}{\beta})^{qk+\frac{qn}{2}}}\leq C\left(\frac{\alpha-q(\alpha-\beta)}{q\alpha\beta x+\alpha}\right)^{\frac{qk+n}{2}}.$$
Taking the $k-$th root of the above inequality and letting  $k\rightarrow+\infty$ we obtain
$$ \frac{1}{(\alpha x+\frac{\alpha}{\beta})^2}\leq \frac{\alpha-q(\alpha-\beta)}{q\alpha\beta x+\alpha}$$ or
\begin{equation}
\label{ako11}
\frac{\beta^2}{(\alpha\beta x+\alpha)^2}\leq  \frac{p\beta-\alpha}{p\alpha(\beta x+1)-\alpha}.
\end{equation}
The inequality \eqref{ako11} can be rewritten in the following form
\begin{equation}
\label{ako111}
\left(\frac{1}{\beta x+1}-\frac{\alpha}{\beta}\right)\left(p-\frac{\alpha}{\beta}-\frac{1}{\beta x+1}\right)\leq 0.
\end{equation}
On the other hand, since we supposed that $p\leq 2$ we can find some $x_{0}>0$ such that $\frac{1}{\beta x_{0}+1}=\frac{p}{2},$ then \eqref{ako111} is equivalent to
$$\left(\frac{p}{2}-\frac{\alpha}{\beta}\right)^{2}\leq 0,$$
i.e. $2\alpha=p\beta.$

{\bf The case $p=1$}

In this case, $P_{\alpha}^{\ast}:L^{\infty}(\mathbb{R}^{n},d\mu_{\beta})\rightarrow L^{\infty}(\mathbb{R}^{n},d\mu_{\beta}).$ First of all, using the same function
$f_{z,k}^{x},$ which is obviously bounded, we may conclude  from \eqref{povratakotpisanih} that $P^{\ast}f_{z,k}^{x}\in L^{\infty}(\mathbb{R}^{n},d\mu_{\beta})$ if $\alpha<\beta.$
For instance, note that $$P_{\alpha}^{\ast}f_{z,k}^{x}(z)=C(k,n)e^{(\frac{1}{\beta}-\frac{1}{\alpha})|z|^2}\frac{|z|^{2k}}{\alpha^{n/2+k}(x+\frac{1}{\beta})^{k+\frac{n}{2}}},$$
where $C(k,n)={\mathcal{Y}_{k}}\left(\frac{z}{|z|},\frac{z}{|z|}\right)=dim(H_{k}(\mathbb{R}^{n})).$

On the other hand, since we supposed that $P_{\alpha}$ is bounded, there is some constant $C>0,$ such that
\begin{equation*}
\|P_{\alpha}^{\ast}f_{z,k}^{x}\|_{\infty}\leq C\|f_{z,k}^{x}\|_{\infty}.
\end{equation*}
So,
\begin{equation}
\label{povratak}
P_{\alpha}^{\ast}f_{z,k}^{x}(z)\leq C\|f_{z,k}^{x}\|_{\infty}\leq C \left(C(k,n)|z|^{k}e^{-\frac{k}{2}}\left(\frac{k}{2x}\right)^{k/2}\right),
\end{equation}
i.e.,
\begin{equation}
\label{povratak1}
\frac{e^{(\frac{1}{\beta}-\frac{1}{\alpha})|z|^2}|z|^{k}}{\alpha^{n/2+k}(x+\frac{1}{\beta})^{k+\frac{n}{2}}}\leq Ce^{-\frac{k}{2}}\left(\frac{k}{2x}\right)^{k/2}.
\end{equation}
Taking the maximal value from the left hand side of the equation \eqref{povratak1} we get
\begin{equation}
\label{povratak11}
\frac{1}{(2(\frac{1}{\alpha}-\frac{1}{\beta}))^{k/2}}\leq C \left(\alpha x+\frac{\alpha}{\beta}\right)^{k+\frac{n}{2}}\left(\frac{1}{2x}\right)^{k/2}.
\end{equation}
Now, similarly to the previous calculations, taking the $k-$th root from the left hand and right hand side in \eqref{povratak11}, and letting $k\rightarrow+\infty$ we obtain
 \begin{equation}
\label{povratak111}
\frac{x}{\frac{1}{\alpha}-\frac{1}{\beta}}\leq  \left(\alpha x+\frac{\alpha}{\beta}\right)^{2}.
\end{equation}
The inequality \eqref{povratak111} is valid for all $x>0$ not belonging to the interval $(\frac{\beta-\alpha}{\alpha\beta},\frac{\alpha}{(\beta-\alpha)\beta})$ which implies that $\frac{\beta-\alpha}{\alpha\beta}=\frac{\alpha}{(\beta-\alpha)\alpha},$ i.e. $2\alpha=\beta.$

{\bf The case $p>2.$}

It is expected to use the duality argument and the previous result. The proof is analogous to the given one in Lemma 2.19 in \cite{zhu}.
However, for the sake of completeness we give the proof for this particular part.

The boundedness of $P_{\alpha}$ on $L^{p}(\mathbb{R}^{n},d\mu_{\beta})$ implies boundedness of $P_{\alpha}^{\ast}$ on $L^{q}(\mathbb{R}^{n},d\mu_{\beta}),$ where $\frac{1}{p}+\frac{1}{q}=1$ and $1<q\leq 2.$
Therefore, the expression
$$\int_{\mathbb{R}^n}e^{-(\frac{1}{\beta}-q(\frac{1}{\beta}-\frac{1}{\alpha}))|z|^2}\left|\int_{\mathbb{R}^n}H_{\alpha}(z,y)\left(f(y)e^{-(\frac{1}{\beta}-\frac{1}{\alpha})|y|^2}\right)d\mu_{\alpha}(y)\right|^{q}dz.$$
is bounded from above by $$C\int_{\mathbb{R}^n}|f(z)|^{q}d\mu_{\beta}(z).$$
 It is clear that  $\frac{1}{\beta}-q\left(\frac{1}{\beta}-\frac{1}{\alpha}\right)>0$ from the previous lemma.

Let us present the function $f(y)=g(y)e^{(\frac{1}{\beta}-\frac{1}{\alpha})|y|^2},$ where $g\in L^{q}(\mathbb{R}^n, d\mu_{\gamma})$
$$\gamma=\left(\frac{1}{\beta}-q\left(\frac{1}{\beta}-\frac{1}{\alpha}\right)\right)^{-1}.$$
So,
$$\int_{\mathbb{R}^n}|P_{\alpha}g|^{q}d\mu_{\gamma}\leq C \int_{\mathbb{R}^n}|g|^{q}d\mu_{\gamma}.$$
 The first part of the proof implies
  $$2\alpha=q\left(\frac{1}{\beta}-q\left(\frac{1}{\beta}-\frac{1}{\alpha}\right)\right)^{-1}$$ which is equivalent with $2\alpha=p\beta.$
\end{proof}

\begin{remark}
The case when $n=2$ deserves a certain comment. Starting from the formula \eqref{red1} it is not hard to see that the following inequality is satisfied
$$\int_{\mathbb{R}^{2}}|H_{\alpha}(x,y)|d\mu_{\beta}(y)\leq 2\int_{\mathbb{R}^{2}}e^{\frac{\left<x,y\right>}{\alpha}}d\mu_{\beta}(y)=2e^{\frac{\beta|x|^2}{4\alpha^2}}.$$
Therefore, it can be easily proved that the condition $\beta=2\alpha$ is sufficient for the boundedness of the operator $P_{\alpha}$ when $p=1.$
Moreover, relaying on the Schur's test, similarly as it was done in analytic case (see \cite{zhu}, Theorem 2.20), it can be shown that the condition $p\beta=2\alpha$ remains in effect also for $p>1.$ Due to the presented fact, in our sequel proves we will not treat this particular case.
\end{remark}

{\bf A proof that the condition $p\beta=2\alpha$  is sufficient }\\

From now on, we will consider the case when the given dimension $n\geq 3$ is an even number, i.e. $n=2N+2,$ $N\geq 1.$

 We fix two positive parameters $\alpha,\beta>0,$ and  define the function $I_{\alpha}^{\beta}$ as follows
\begin{equation}
\label{funkcija}
I_{\alpha}^{\beta}(y)=\int_{\mathbb{R}^{n}}|H_{\alpha}(y,x)|d\mu_{\beta}(x).
\end{equation}
In the Lemma \ref{lemmma} we give a certain upper estimate of the function $I_{\alpha}^{\beta}(y)$ which will play a key role in proving the Theorem \ref{dovoljan} and Theorem \ref{dovoljan1}.
\begin{lemma}
\label{lemmma}
For the  function $I_{\alpha}^{\beta}$ defined in \eqref{funkcija}
and positive number $\epsilon,$ such that $0<\epsilon<1$  the following estimate holds
$$I_{\alpha}^{\beta}(y)<{}_1C_{\alpha,\beta}^{N,n}(\epsilon)\Psi(|y|)+{}_2C_{\alpha,\beta}^{N,n}\Phi(|y|),\enspace|y|>0,$$
where
\begin{eqnarray*}
\Psi(|y|)&=&\sum_{j=0}^{N}{\left(N\atop j\right)}\frac{N!}{j!}\frac{\beta^{\frac{N+j+3}{2}}\Gamma(\frac{N+j+3}{2})}{\alpha^{j}|y|^{N-j-1}}
   {}_1F_{1}\left(\frac{N+j+3}{2}, \frac{3}{2}; \frac{(1-\epsilon^{2})\beta |y|^2}{4 \alpha^2}\right)\\
   \Phi(|y|)&=&\Gamma\left(\frac{n}{2}\right)\sum_{j=0}^{N}{\left(N\atop j\right)}\frac{N!}{j!}{}_1F_{1}\left(\frac{n}{2}, \frac{1}{2}, \frac{\theta^2\beta  |y|^2}{
     4 \alpha^2}\right)\\
     & +&\Gamma\left(\frac{n+1}{2}\right)\sum_{j=0}^{N}{\left(N\atop j\right)}\frac{N!}{j!}\frac{\beta^{\frac{1}{2}}}{\alpha}
    \theta |y|  {}_1F_{1}\left(\frac{1 + n}{2}, \frac{3}{2}, \frac{
     \theta^2\beta  |y|^2}{4 \alpha^2}\right).
\end{eqnarray*}
and
   $${}_1C_{\alpha,\beta}^{N,n}(\epsilon)=\frac{2^{3-2N}\pi^{-1/2}\alpha^{N}\beta^{-n/2}\prod_{k=1}^{n-3}\frac{\Gamma(\frac{1+k}{2})}{\Gamma(1+\frac{k}{2})}}{\epsilon^{N+1}(1-\epsilon)(N-1)!}\max_{0\leq j\leq N,0\leq k\leq j}\int_{-1}^{1}|G_{N}(t)t^{j-k}|dt,$$
   $${}_2C_{\alpha,\beta}^{N,n}=\frac{2^{1-2N}\pi^{-n/2}\alpha^{N-1}\prod_{k=1}^{n-3}\frac{\Gamma(\frac{1+k}{2})}{\Gamma(1+\frac{k}{2})}}{(N-1)!}\max_{0\leq j\leq N,0\leq k\leq j}\int_{-1}^{1}|G_{N}(t)t^{j-k}|dt$$
\end{lemma}

\begin{proof}

 According to the formula \eqref{miroslav}, we have
\begin{equation}
\label{lanovita}
\begin{split}
H_{\alpha}(x,y)&=\frac{2^{1-2N}}{(N-1)!}\left(\frac{\alpha}{iV}\right)^{N-1}\sum_{j=0}^{N}{\left(N\atop j\right)}\frac{N!}{j!}\\
&e^{\frac{\left<x,y\right>}{\alpha}}\sum_{k=0}^{j}\left(j\atop k\right)\frac{(\left<x,y\right>)^{k}(iV)^{j-k}}{\alpha^{j}}\int_{-1}^{1}G_{N}(t)t^{j-k}e^{\frac{itV}{\alpha}}dt,
\end{split}
\end{equation}
where $G_{N}(t)=(-1)^{N-1}\frac{\partial^{N-1}}{\partial t^{N-1}}(1-t^2)^{N-1}$ and as it was before, for the sake of brevity, we write $V$ for $V(x,y)=\sqrt{|x|^2|y|^2-\left<x,y\right>^2}.$

Then,
\begin{equation}
|H_{\alpha}(x,y)|\leq C(N)e^{\frac{\left<x,y\right>}{\alpha}}\sum_{j=0}^{N}{\left(N\atop j\right)}\frac{N!}{j!}\sum_{k=0}^{j}\left(j\atop k\right)\frac{|\left<x,y\right>|^{k}}{\alpha^{j}V^{N-j+k-1}},
\end{equation}
where
$$C(N)=\frac{2^{1-2N}\alpha^{N-1}}{(N-1)!}\max_{0\leq j\leq N,0\leq k\leq j}\int_{-1}^{1}|G_{N}(t)t^{j-k}|dt.$$

At this point, we introduce the functions $\{I_{j}^{k}\}$ defined as follows
 $$I_{j}^{k}(y)=(\pi\beta)^{-n/2}\int_{\mathbb{R}^{n}}e^{\frac{\left<x,y\right>}{\alpha}}\frac{|\left<x,y\right>|^{k}}{V^{N+k-j-1}}e^{-\frac{|x|^2}{\beta}}dx,$$
 $j\in\{0,1,...,N\},$ $k\in\{0,1,...,j\}.$

  Note that
\begin{equation}
\label{lagano}
\begin{split}
  I_{N-1}^{0}(f)&=(\pi\beta)^{-n/2}\int_{\mathbb{R}^{n}}e^{\frac{\left<x,y\right>}{\alpha}}e^{-\frac{|x|^2}{\beta}}dx\\
  &=(\pi\beta)^{-n/2}\prod_{i=1}^{n}\int_{\mathbb{R}}e^{\frac{x_i y_{i}}{\alpha}}e^{-\frac{x_{i}^{2}}{\beta}}dx_{i}\\
  &=e^{\frac{\beta|y|^{2}}{4\alpha^{2}}}.
  \end{split}
\end{equation}

For the given positive number $\epsilon$ such that $0<\epsilon<1,$ and the fixed vector $y\in \mathbb{R}^{n},$ $y\neq 0,$ let us denote by $$A_{\epsilon}^{y}=\left\{x\in\mathbb{R}^{n}|\frac{|\left<x,y\right>|}{|x||y|}\geq\sqrt{1-\epsilon^{2}}\right\}.$$

Then,
\begin{equation}
\begin{split}
I_{j}^{k}(y)(y)=&(\pi\beta)^{-n/2}\int_{\mathbb{R}^{n}\setminus A_{\epsilon}^{y}}e^{\frac{\left<x,y\right>}{\alpha}}\frac{|\left<x,y\right>|^{k}}{V^{N+k-j-1}}e^{-\frac{|x|^2}{\beta}}dx\\
&+(\pi\beta)^{-n/2}\int_{A_{\epsilon}^{y}}e^{\frac{\left<x,y\right>}{\alpha}}\frac{|\left<x,y\right>|^{k}}{V^{N+k-j-1}}e^{-\frac{|x|^2}{\beta}}dx\\
&\leq \frac{(\pi\beta)^{-n/2}}{\epsilon^{N+k-j+1}}\frac{1}{|y|^{N-j-1}}\int_{\mathbb{R}^{n}}e^{\frac{\left<x,y\right>}{\alpha}}\frac{e^{-\frac{|x|^2}{\beta}}}{|x|^{N-j-1}}dx\\
&+(\pi\beta)^{-n/2}\int_{A_{\epsilon}^{y}}e^{\frac{\left<x,y\right>}{\alpha}}\frac{|\left<x,y\right>|^{k}}{V^{N+k-j-1}}e^{-\frac{|x|^2}{\beta}}dx.\\
\end{split}
\end{equation}
Further, we denote by
\begin{equation}
\label{prvi}
{}_1I_{j}^{k}(y)=\frac{1}{|y|^{N-j-1}}\int_{\mathbb{R}^{n}}e^{\frac{\left<x,y\right>}{\alpha}}\frac{e^{-\frac{|x|^2}{\beta}}}{|x|^{N-j-1}}dx
\end{equation}
and
\begin{equation}
\label{prvi1}
{}_2I_{j}^{k}(y)=\int_{A_{\epsilon}^{y}}e^{\frac{\left<x,y\right>}{\alpha}}\frac{|\left<x,y\right>|^{k}}{V^{N+k-j-1}}e^{-\frac{|x|^2}{\beta}}dx.
\end{equation}
The integral in \eqref{prvi} can be computed  using the change of variables $x=T^{-1}x'$ provided by the orthogonal transformation $T$ of $\mathbb{R}^{n}$ such that $Ty=|y|e_1,$ where $e_1=(1,0,...,0).$

Therefore,
\begin{equation}\label{prvi2}
\begin{split}
 {}_1I_{j}^{k}(y)&=\frac{1}{|y|^{N-j-1}}\int_{\mathbb{R}^{n}\setminus A_{\epsilon}^{y}}e^{\frac{\left<x,y\right>}{\alpha}}\frac{e^{-\frac{|x|^2}{\beta}}}{|x|^{N-j-1}}dx\\
&=\frac{1}{|y|^{N-j-1}}\int_{\mathbb{R}^{n}\setminus A_{\epsilon}^{y}}e^{\frac{\left<x',Ty\right>}{\alpha}}\frac{e^{-\frac{|x'|^2}{\beta}}}{|x'|^{N-j-1}}dx'.
\end{split}
\end{equation}
Now, inserting the $n-$dimensional spherical coordinates in the last integral of \eqref{prvi2},
where $x_{1}'=r\cos{\phi_1}, \phi_{1}\in[\arccos{\sqrt{1-\epsilon^2}},\pi-\arccos{\sqrt{1-\epsilon^2}}),$
with the Jacobian $\left|\frac{D(x_1,x_2,...,x_n)}{D(r,\phi_1,\phi_2,...,\phi_{n-1})}\right|=r^{n-1}\sin^{n-2}{\phi_1}\sin^{n-3}{\phi_2}\cdot\cdot\cdot\sin{\phi_{n-1}},$ we get
\begin{equation*}
\begin{split}
{}_1I_{j}^{k}(y)&=\frac{C(n)}{|y|^{N-j-1}}\int_{0}^{\infty}e^{-\frac{r^2}{\beta}}r^{N+j+2}\int_{\arccos{\sqrt{1-\epsilon^2}}}^{\pi-\arccos{\sqrt{1-\epsilon^2}}}e^\frac{r|y|\cos{\phi_{1}}}{\alpha}\sin^{n-2}{\phi_1}d\phi_1dr\\
&\leq\frac{C(n)}{|y|^{N-j-1}}\int_{0}^{\infty}e^{-\frac{r^2}{\beta}}r^{N+j+2}\int_{\arccos{\sqrt{1-\epsilon^2}}}^{\pi-\arccos{\sqrt{1-\epsilon^2}}}e^\frac{r|y|\cos{\phi_{1}}}{\alpha}\sin{\phi_1}d\phi_1dr\\
&=\frac{2C(n)}{|y|^{N-j-1}}\int_{0}^{\infty}e^{-\frac{r^2}{\beta}}r^{N+j+2}\frac{ \sinh(\frac{r|y|}{\alpha}  \sqrt{1 - \epsilon^2})}{|y|\alpha^{-1}}dr\\
&= C(n)\frac{2\alpha\beta^{\frac{N+j+3}{2}}\Gamma(\frac{N+j+3}{2})}{|y|^{N-j-1}}
   {}_1F_{1}\left(\frac{N+j+3}{2}, \frac{3}{2}; \frac{(1-\epsilon^{2})\beta |y|^2}{4 \alpha^2}\right),
\end{split}
\end{equation*}
where $C(n)=2\pi^{\frac{n-1}{2}}\prod_{k=1}^{n-3}\frac{\Gamma(\frac{1+k}{2})}{\Gamma(1+\frac{k}{2})}.$

Similarly, the set $A_{\epsilon}^{y}$ presents a certain $n-$dimesional cone "in direction" of vector $y.$ Thus, the set $T(A_{\epsilon}^{y})=A_{\epsilon}^{1}$ presents a new cone "in direction" of $e_{1}.$

Using the same change of variable $x=T^{-1}x',$ we get
\begin{equation}
\begin{split}
\label{konus}
\int_{A_{\epsilon}^{1}}e^{\frac{\left<T^{-1}x',y\right>}{\alpha}}\frac{|\left<T^{-1}x',y\right>|^{k}}{V^{N+k-j-1}}e^{-\frac{|T^{-1}x'|^2}{\beta}}&dx'\\
&=\int_{A_{\epsilon}^{1}}e^{\frac{\left<x' ,Ty\right>}{\alpha}}\frac{|\left<x',Ty\right>|^{k}}{V^{N+k-j-1}}e^{-\frac{|x'|^2}{\beta}}dx'.\\
\end{split}
\end{equation}
Now, inserting the $n-$dimensional spherical coordinates  in the right-hand side of the equation \eqref{konus}, where $\phi_{1}\in[0,\arccos(\sqrt{1-\epsilon^2})],$ we get
\begin{equation}
\label{konus1}
\begin{split}
&\int_{A_{\epsilon}^{1}}e^{\frac{\left<x' ,Ty\right>}{\alpha}}\frac{|\left<x',Ty\right>|^{k}}{V^{N+k-j-1}}e^{-\frac{|x'|^2}{\beta}}dx'\\
&=C'(n)\int_{0}^{\infty}e^{-\frac{r^2}{\beta}}r^{n-1}\int_{0}^{\arccos(\sqrt{1-\epsilon^2})}e^{\frac{r|y|\cos{\phi_1}}{\alpha}}\cos^{k}{\phi_1}\sin^{N+j-k+1}{\phi_1}d\phi_{1}dr\\
&\leq C'(n)\int_{0}^{\infty}e^{-\frac{r^2}{\beta}}r^{n-1}\int_{0}^{\arccos(\sqrt{1-\epsilon^2})}e^{\frac{r|y|}{\alpha}\cos{\phi_1}}\sin{\phi_1}d\phi_{1}dr\\
&=C'(n)\int_{0}^{\infty}e^{-\frac{r^2}{\beta}}r^{n-1}\frac{e^{\frac{r|y|}{\alpha}}-e^{\sqrt{1-\epsilon^{2}}\frac{r|y|}{\alpha}}}{r|y|\alpha^{-1}}dr\\
&=C'(n)\int_{0}^{\infty}e^{-\frac{r^2}{\beta}}e^{\frac{r\theta|y|}{\alpha}}r^{3}dr\\
&= C'(n)\frac{\beta^{n/2}}{2} \Gamma\left(\frac{n}{2}\right) {}_1F_{1}\left(\frac{n}{2}, \frac{1}{2}, \frac{\theta^2\beta  |y|^2}{
     4 \alpha^2}\right)\\
& +C'(n)\frac{\beta^{\frac{n+1}{2}}}{2\alpha}
    \theta |y| \Gamma\left(\frac{n+1}{2}\right) {}_1F_{1}\left(\frac{1 + n}{2}, \frac{3}{2}, \frac{
     \theta^2\beta  |y|^2}{4 \alpha^2}\right).
\end{split}
\end{equation}
where $C'(n)=2C(n),$ and $\theta\in(\sqrt{1-\epsilon^{2}},1).$

\end{proof}
\begin{remark}
The expressions for the functions $\Psi$ and $\Phi,$ as well for the constants ${}_1C_{\alpha,\beta}^{N,n}(\epsilon)$ and ${}_2C_{\alpha,\beta}^{N,n},$ which occur in the Lemma \ref{lemmma} seem to be a little bit  unsuitable for further computations and estimation of the norm for of the operator $P_{\alpha}.$ However, the essential information provided by the Lemma \ref{lemmma} is the presence of the factor $s=(1-\epsilon^{2})$ ($s=\theta$) in the argument of the function ${}_1F_{1}(\cdot,\cdot,\cdot;\frac{s\beta|y|^2}{4\alpha^{2}})$
which, in fact, will resolve the problem of proving  the sufficiency  of the condition $p\beta=2\alpha.$
\end{remark}
\begin{remark}
\label{lemma11}
From the Lemma \ref{lemmma} we may conclude that there are some constants $A_{1}(N,n,\alpha,\beta)$ and $A_{2}(N,n,\alpha,\beta)$ such that
\begin{equation}\label{jednacina1234}\begin{split}
I_{\alpha}^{\beta}(y)\leq& A_{1}(N,n,\alpha,\beta,\theta) \Delta(|y|){}_1F_{1}\left(\frac{1 + n}{2}, \frac{3}{2}, \frac{
     \theta^2\beta  |y|^2}{4 \alpha^2}\right)\\
     &+A_{2}(N,n,\alpha,\beta) {}_1F_{1}\left(\frac{n}{2}, \frac{1}{2}, \frac{\theta^2\beta  |y|^2}{
     4 \alpha^2}\right),
     \end{split}
\end{equation}
where
$$\Delta(|y|)=\left\{\begin{array}{rl}
|y|+|y|^{1-N},&0<|y|< 1,\\
|y|,&|y|>1
\end{array}
\right.$$ and $\theta\in (\sqrt{1-\epsilon^2}, 1).$
\end{remark}
\begin{theorem}
\label{dovoljan}
If $\beta=2\alpha,$ then the operator $P_{\alpha}:L^{1}(\mathbb{R}^{n},d\mu_{\beta})\rightarrow L^{1}(\mathbb{R}^{n},d\mu_{\beta})$ is bounded.
\end{theorem}

\begin{proof}
First of all, for $f\in L^{1}(\mathbb{R}^{n},d\mu_{\beta})$  Fubini's theorem implies
\begin{equation}
\|P_{\alpha}f\|_{L^{1}(\mathbb{R}^{n},d\mu_{\beta})\rightarrow L^{1}(\mathbb{R}^{n},d\mu_{\beta})}\\
\leq C(\alpha,\beta) \int_{\mathbb{R}^{n}}|f(y)|e^{-\frac{|y|^2}{2\alpha}}I_{\alpha}^{\beta}(y)d\mu_{\beta}(y),
\end{equation}
where $I_{\alpha}^{\beta}(y)$ is the function from the Lemma \ref{lemmma} and $C(\alpha,\beta)$ is a constant depending from $\alpha$ and $\beta.$

The integral $\int_{\mathbb{R}^{n}}|f(y)|e^{-\frac{|y|^2}{2\alpha}}I_{\alpha}^{\beta}(y)d\mu_{\beta}(y)$ can be rewritten in form
\begin{equation}
\label{misli}
\int_{|y|<1}|f(y)|e^{-\frac{|y|^2}{2\alpha}}I_{\alpha}^{\beta}(y)d\mu_{\beta}(y)+\int_{|y|>1}|f(y)|e^{-\frac{|y|^2}{2\alpha}}I_{\alpha}^{\beta}(y)d\mu_{\beta}(y).
\end{equation}
According to the inequality \eqref{otkrice}, there is some positive constant $M$ such that
$$\int_{|y|<1}|f(y)|e^{-\frac{|y|^2}{2\alpha}}I_{\alpha}^{\beta}(y)d\mu_{\beta}(y)\leq M\int_{|y|<1}|f(y)|d\mu_{\beta}(y)\leq M\|f\|_{L^{1}(\mathbb{R}^{n},d\mu_{\beta})},$$
while, since $0<\epsilon<1$ and $\theta\in (\sqrt{1-\epsilon^2}, 1),$ the Lemma \ref{lemmma} and asymptotic behaviour of Kummer's confluent function ${}_1F_{1}(a,b;x)$ for a large argument $x$ given in \eqref{ast1234}, imply that there is a constant $C>0$ such that
$$\sup_{|y|>1}e^{-\frac{|y|^2}{2\alpha}}I_{\alpha}^{\beta}(|y|)\leq C,$$ which together with \eqref{misli} gives
$$\|P_{\alpha}f\|_{L^{1}(\mathbb{R}^{n},d\mu_{\beta})\rightarrow L^{1}(\mathbb{R}^{n},d\mu_{\beta})}\leq \tilde{C}\|f\|_{L^{1}(\mathbb{R}^{n},d\mu_{\beta})}$$ where the constant $\tilde{C}$ is provided from the previously determined constants.
\end{proof}

In the sequel, we prove that the operator $P_{\alpha}$ is bounded on $L^{p}(\mathbb{R}^{n},d\mu_{\beta})$ for $1<p<\infty$ by appealing to the well known Schur's test (see, for instance, \cite{zhu1}, Theorem 3.6).
\begin{lemma}
\label{Schur}
Suppose that $(X,\mu)$ is a $\sigma-$finite measure space and $K(x,y)$ is a nonnegative measurable function on $X\times X$ and $T$ is associated integral operator
$$Tf(x)=\int_{X}K(x,y)f(y)d\mu(y).$$
Let $1<p<\infty$ and $\frac{1}{p}+\frac{1}{q}=1.$ If there exist a positive constant $C$ and a positive function $h$ on $X$ such that
$$\int_{X}K(x,y)h^{p}(y)d\mu(y)\leq C h^{p}(x),$$
for almost every $x$ in $X$ and
$$\int_{X}K(x,y)h^{q}(x)d\mu(x)\leq C h^{q}(y),$$
for almost every $y$ in $X,$ then $T$ is bounded on $L^{p}(X,d\mu)$ with $\|T\|\leq C.$
\end{lemma}

\begin{theorem}
\label{dovoljan1}
Let $\alpha,\beta>0.$ If $p\beta=2\alpha, p>1,$ then the operator $P_{\alpha}$ is bounded on $L^{p}(\mathbb{R}^{n},d\mu_{\beta}).$
\end{theorem}
\begin{proof}
Suppose $1<p\leq2.$

Let $\frac{1}{p}+\frac{1}{q}=1$ and consider the positive function
$$h(y)=e^{\delta|y|^2}, y\in\mathbb{R}^{n},$$
where $\delta$ is going to be determined latter.

We observe the operator $Q_{\alpha}^{\beta}:L^{p}(\mathbb{R}^{n}d\mu_{\beta})\rightarrow L^{p}(\mathbb{R}^{n},d\mu_{\beta})$ defined by
$$Q_{\alpha}^{\beta}f(x)=\int_{\mathbb{R}^{n}}K_{\alpha}(x,y)f(y)d\mu_{\beta}(x),$$
where $K_{\alpha}(x,y)=e^{(\frac{1}{\beta}-\frac{1}{\alpha})|y|^2}|H_{\alpha}(x,y)|$ is a positive kernel.

We  consider the integral
\begin{equation}
\label{uvod}
\int_{\mathbb{R}^{n}}K_{\alpha}(x,y)h^{p}(y)d\mu_{\beta}(y)=(\pi\beta)^{-n/2}\int_{\mathbb{R}^{n}}|H_{\alpha}(x,y)|h^{p}(y)d\mu_{\alpha}(y),x\in\mathbb{R}^{n}.
\end{equation}
We suppose that $p\delta<\frac{1}{\alpha},$ which means  that the right-hand side of \eqref{uvod} is
$\left(\frac{\gamma}{\beta}\right)^{n/2}I_{\alpha}^{\gamma}(x),$ where $\gamma=\left(\frac{1}{\alpha}-p\delta\right)^{-1}.$

 Recalling the inequality \eqref{otkrice} for $|x|<A$ ($A>1$) we can surely find a constant $C_1>0$ such that
\begin{equation}\label{akoje}I_{\alpha}^{\gamma}(x)\leq C_{1}e^{p\delta|x|^2},|x|<A\end{equation} while taking that $\frac{\gamma}{4\alpha^2}=p\delta$ and using the inequality \eqref{jednacina1234} from the Remark \ref{lemma11} and relation \eqref{otkrice} we have
\begin{equation}
\label{dosta}
\begin{split}
&\lim_{|x|\rightarrow+\infty}\frac{I_{\alpha}^{\gamma}(x)}{e^{p\delta |x|^2}}\\
&\leq \lim_{|x|\rightarrow+\infty} A_{1}(N,n,\alpha,\beta,\theta) |x|{}_1F_{1}\left(\frac{1 + n}{2}, \frac{3}{2},
     \theta^2p\delta |x|^{2}\right)e^{-p\delta|x|^2}\\
& +\lim_{|x|\rightarrow+\infty}A_{2}(N,n,\alpha,\beta) {}_1F_{1}\left(\frac{n}{2}, \frac{1}{2},  \theta^2p\delta |x|^{2}\right)e^{-p\delta|x|^2}=0.\end{split}
    \end{equation}
    Summing together \eqref{akoje} and \eqref{dosta} we proved the existence of some constant $C>0$ which satisfies the following inequality
    $$ \left(\frac{\gamma}{\beta}\right)^{n/2}I_{\alpha}^{\gamma}(x)\leq C e^{p\delta|x|^2},\enspace x\in \mathbb{R}^{n}.$$
Further, we treat the integral
\begin{equation}\label{akoje1}e^{(\frac{1}{\beta}-\frac{1}{\alpha})|y|^2}\int_{\mathbb{R}^{n}}|H_{\alpha}(x,y)|e^{q\delta|x|^2}d\mu_{\beta}(x), y\in \mathbb{R}^{n}.\end{equation}
 Taking that $q\delta<\frac{1}{\beta}$  the integral \eqref{akoje1} becomes $\left(\frac{s}{\beta}\right)^{n/2}e^{(\frac{1}{\beta}-\frac{1}{\alpha})|y|^2}I_{\alpha}^{s}(y),$ where $s=(\frac{1}{\beta}-q\delta)^{-1}.$

 Repeating the previous procedure by  choosing that $\frac{s}{4\alpha^2}=q\delta+\frac{1}{\alpha}-\frac{1}{\beta}$ we conclude that there is some constant $C>0$ such that $$\left(\frac{s}{\beta}\right)^{n/2}e^{(\frac{1}{\beta}-\frac{1}{\alpha})|y|^2}I_{\alpha}^{s}(y)\leq Ce^{q\delta|y|^2}, y\in \mathbb{R}^{n}.$$
Collecting all the relations  between $\alpha, \beta,\delta$ and $ p,q$ we get the equation

$$\left(\frac{1}{\beta}-q\delta\right)\left(q\delta+\frac{1}{\alpha}-\frac{1}{\beta}\right)=\left(\frac{1}{\alpha}-p\delta\right)p\delta$$
 which gives two   solutions for $\delta.$

Namely, we have
$$\delta_1=\frac{-\alpha+\beta}{\alpha\beta(p-q)}, \delta_{2}=\frac{1}{\beta(p+q)}.$$
Since $\alpha\leq \beta,$ we chose  $$\delta=\frac{1}{\beta(p+q)}=\frac{1}{pq\beta}.$$
Finally, if $p>2,$ let $f\in L^{p}(\mathbb{R}^{n},d\mu_{\beta})$ and $g\in L^{q}(\mathbb{R}^{n},d\mu_{\beta}).$ Then, $g(x)=h(x)e^{(\frac{1}{\beta}-\frac{1}{\alpha})|x|^2},$ where $h\in L^{q}(\mathbb{R}^{n},d\mu_{\lambda})$ and $\lambda=(\frac{1}{\beta}-q(\frac{1}{\beta}-\frac{1}{\alpha})))^{-1}.$ Precisely, the condition $p\beta=2\alpha$ implies $\lambda=\frac{2\alpha}{q}.$

Moreover, we have
$$\|g\|_{L^{q}(\mathbb{R}^{n},d\mu_{\beta})}=\left(\frac{\lambda}{\beta}\right)^{\frac{n}{2q}}\|h\|_{L^{q}(\mathbb{R}^{n},d\mu_{\lambda})}.$$
Since $q\lambda=2\alpha$ and $1<q\leq2,$ the operator $Q_{\alpha}^{\lambda}:L^{q}(\mathbb{R}^{n},d\mu_{\lambda})\rightarrow L^{q}(\mathbb{R}^{n},d\mu_{\lambda})$ is bounded by the first part of the proof. Bearing this in mind, we get the sequel sequence of inequalities
\begin{equation}\label{duality1}\begin{split}
\left|\left<P_{\alpha}f,g\right>\right|&=\left|\int_{\mathbb{R}^{n}}(P_{\alpha}f)(x)\overline{g(x)}d\mu_{\beta}(x)\right|\\
&\leq \int_{\mathbb{R}^{n}}|f(y)|e^{(\frac{1}{\beta}-\frac{1}{\alpha})|y|^{2}}\left(\int_{\mathbb{R}^{n}}|H_{\alpha}(x,y)||g(x)|d\mu_{\beta}(x)\right)d\mu_{\beta}(y)\\
&=\left(\frac{\lambda}{\beta}\right)^{\frac{n}{2}}\int_{\mathbb{R}^{n}}|f(y)|\left(e^{(\frac{1}{\beta}-\frac{1}{\alpha})|y|^{2}}Q_{\alpha}^{\lambda}(|h|)(y)\right)d\mu_{\beta}(y)\\
&\leq C \|f\|_{L^{p}(\mathbb{R}^{n},d\mu_{\beta})}\|g\|_{L^{q}(\mathbb{R}^{n},d\mu_{\beta})}.
\end{split}
\end{equation}
Now, it is clear that $$\|P_{\alpha}\|_{L^{p}(\mathbb{R}^{n},d\mu_{\beta}\rightarrow L^{p}(\mathbb{R}^{n},d\mu_{\beta}))}\leq C,$$
where the constant $C$ is from \eqref{duality1} guaranteed by the boundedness of the operator $Q_{\alpha}^{\lambda}.$
\end{proof}
\section{The Hilbert case $P_{\alpha}:L^{2}(\mathbb{R}^{n},d\mu_{\beta})\rightarrow L^{2}(\mathbb{R}^{n},d\mu_{\beta})$}
 In this section we observe the set of all polynomials $P(\mathbb{R}^{n})$ considered as a subspace in $L^{2}(\mathbb{R}^{n},d\mu_{\beta})$ (we denote it by $\Pi_{\beta}$). Despite the fact that the operator $P_{\alpha}:L^{2}(\mathbb{R}^{n},d\mu_{\beta})\rightarrow L^{2}(\mathbb{R}^{n},d\mu_{\beta})$ is bounded only in case when it comes $\alpha=\beta,$ in Proposition \ref{polinomx} we show that the operator $P_{\alpha}$   is bounded on space $\Pi_{\beta}.$

 Since $\Pi_{\beta}$ is dense in $L^{2}(\mathbb{R}^{n},d\mu_{\beta}),$ the Theorem \ref{main} asserts that a continuous (bounded) extension of the operator $P_{\alpha}$ on entire space $L^{2}(\mathbb{R}^{n},d\mu_{\beta})$ is possible only for $\alpha=\beta.$

In Lemma \ref{donjalema} we want to state explicitly one auxiliary result (which was already used in certain form) related to the question of the image set for the operator $P_{\alpha}$ on a set of all polynomials.

\begin{lemma}
\label{donjalema}
Let $f(y_1,...,y_n)=\sum_{k\geq0}\sum_{|s|=k}a_{s}y^{s}$ be a given polynomial. Then
\begin{equation}
\label{prva1}
P_{\alpha}f(w)=\sum_{k\geq0}\sum_{i\geq 0}\psi_{k}^{i}(w).\end{equation}
Here, we denote by $\psi_{k}^{i}(y)=\sum_{|s|=i}a_{s}^{k}y^{s}$  the harmonic polynomials of degree $i$  which appear in the decomposition of the polynomial $\sum_{|s|=k}a_{s}y^{s}$ regarding \eqref{polinom11}.
\end{lemma}
\begin{proof}
\begin{equation*}
\begin{split}
&P_{\alpha}f(w)\\
&=(\pi\alpha)^{-n/2}\sum_{k\geq0}\sum_{i=0}^{\infty}\frac{1}{\alpha^{i}\left(\frac{n}{2}\right)_{i}}\int_{\mathbb{R}^n}Y_{i}(w,y)e^{-\frac{|y|^2}{\alpha}}f(y)dy\\
&=\sum_{k\geq0}\sum_{i=0}^{\infty}\frac{|w|^{i}(\pi\alpha)^{-n/2}}{\alpha^{i}\left(\frac{n}{2}\right)_{i}}\\
&\times\sum_{s\geq 0}\int_{0}^{\infty}e^{-\frac{r^{2}}{\alpha}}r^{n+i+s-1}dr\int_{\mathbb{S}^{n-1}}\mathcal{Y}_{i}\left(\frac{w}{|w|},\xi\right)\psi_{k}^{s}(\xi)d\sigma(\xi)\\
&=\sum_{k\geq0}\sum_{i\geq 0}\frac{|w|^{i}(\pi\alpha)^{-n/2}}{\alpha^{i}\left(\frac{n}{2}\right)_{i}}\\
&\times\int_{0}^{\infty}e^{-\frac{r^{2}}{\alpha}}r^{n+2i-1}dr\int_{\mathbb{S}^{n-1}}\mathcal{Y}_{i}\left(\frac{w}{|w|},\xi\right)\psi_{k}^{i}(\xi)d\sigma(\xi)\\
&=\sum_{k\geq0}\sum_{i\geq 0}\psi_{k}^{i}(w).
\end{split}
\end{equation*}
\end{proof}
\begin{proposition}
\label{polinomx}
Let $\beta,\alpha>0,$ the operator $P_{\alpha}$ is bounded on $\Pi_{\beta},$ and
$$\|P_{\alpha}\|_{\Pi_{\beta}\rightarrow L_{\beta}^{2}(\mathbb{R}^{n})}\leq C(n,\beta),$$ where
$$C(n,\beta)\leq\left\{\begin{array}{rl}
1,&\beta\geq 1,\enspace \mbox{or}\enspace\beta^{-1}\leq\frac{n}{2}\\
(\frac{2}{n\beta})^{\frac{[\beta^{-1}-\frac{n}{2}+1]+1}{2}},&\beta^{-1}>\frac{n}{2}
\end{array}
\right.$$
\end{proposition}
\begin{proof}
  Let    $f\in P(\mathbb{R}^{n}),$ then $f(x)=\sum_{k\geq0}p_{k}(x),$ where $p_{k}\in P_{k}(\mathbb{R}^{n})$ is a homogenous polynomial of degree $k.$
  According to the Lemma \ref{donjalema} we have
$$P_{\alpha}f(x)=\sum_{k\geq0}\sum_{i\geq 0}\psi_{k}^{i}(x).$$
We determine the polynomials $\psi_{k}^{i}(w)$ as follows
$\psi_{k}^{i}(x)=\sum_{|s|=i}a_{s}^{k}x^{s}.$
Using the Theorem 5.14 from \cite{Axler}, we obtain the sequel identities
\begin{equation}
\label{pocetna}
\begin{split}
&\|P_{\alpha}f\|_{L^{2}(\mathbb{R}^{n},d\mu_{\beta})}^{2}\\
&=\int_{\mathbb{R}^{n}}|\sum_{k\geq0}\sum_{i\geq 0}\psi_{k}^{i}(x)|^{2}d\mu_{\beta}(x)\\
&=(\beta)^{-n/2}\sum_{k\geq0}\sum_{j\geq0}\sum_{i\geq 0}\int_{0}^{\infty}e^{-\frac{r^2}{\beta}}r^{n+2i-1}dr\sum_{|s|=i}a_{s}^{k}\overline{a_{s}^{j}}\frac{s!}{2^{i-1}\Gamma(\frac{n}{2}+i)}\\
&=\sum_{k\geq0}\sum_{j\geq0}\sum_{i\geq0}\frac{\beta^{i}}{2^{i}}\sum_{|s|=i}a_{s}^{k}\overline{a_{s}^{j}}s!.\\
\end{split}
\end{equation}

On the other hand,
\begin{equation*}
\label{pocetak1}
\begin{split}
&\|f\|_{L^{2}(\mathbb{R}^{n},d\mu_{\beta})}^{2}\\
&=(\pi\beta)^{-n/2}\int_{0}^{\infty}e^{-\frac{r^{2}}{\beta}}r^{n-1}\int_{\mathbb{S}^{n-1}}|\sum_{k\geq0}r^{k}\sum_{i\geq0}\psi_{k}^{i}(\xi)|^{2}d\sigma(\xi)dr\\
&=\sum_{k\geq0}\sum_{j\geq0}\beta^{\frac{k+j}{2}}\Gamma\left(\frac{k+j+n}{2}\right)\sum_{i\geq0}\frac{1}{2^{i}\Gamma(i+\frac{n}{2})}\sum_{|s|=i}a_{s}^{k}\overline{a_{s}^{j}}s!.
\end{split}
\end{equation*}
Now, we can see that
$$\|P_{\alpha}f\|_{L^{2}(\mathbb{R}^{n},d\mu_{\beta})}^{2}\leq (C(n,\beta))^{2} \|f\|_{L^{2}(\mathbb{R}^{n},d\mu_{\beta})}^{2},$$
where $$(C(n,\beta))^{2}=\max_{k,j\in\mathbb{N}_{0}, i\leq k,i\leq j}\frac{\beta^{i-\frac{k+j}{2}}\Gamma(i+\frac{n}{2})}{\Gamma(\frac{k+j+n}{2})}.$$

Further,
$$(C(n,\beta))^{2}\leq\left\{\begin{array}{rl}
1,&\beta\geq 1,\enspace \mbox{or}\enspace\beta^{-1}\leq\frac{n}{2}\\
(\frac{2}{n\beta})^{[\beta^{-1}-\frac{n}{2}+1]+1},&\beta^{-1}>\frac{n}{2}
\end{array}
\right.$$

\end{proof}

\end{document}